\documentclass[12pt,amscd]{amsart}
\usepackage[all]{xy}
\usepackage{graphicx}
\usepackage{amsmath,amsxtra,amssymb,latexsym, amscd,amsthm,enumerate}
\usepackage{indentfirst}
\usepackage[mathscr]{eucal}
\usepackage{stackrel}
\newtheorem{thm}{Theorem}[section]
\newtheorem{cor}[thm]{Corollary}
\newtheorem{lem}[thm]{Lemma}
\newtheorem{prop}[thm]{Proposition}

\theoremstyle{definition}

\newtheorem{exam}[thm]{Example}

\DeclareMathOperator{\m}{m}

\begin{document}

\title[Coverings, Matchings and the number of maximal independent sets of  graphs]{Coverings, Matchings and the number of maximal independent sets of  graphs}

\author{Do  Trong Hoang}

\email{dthoang@math.ac.vn}
\address{Institute of Mathematics, VAST, 18 Hoang Quoc Viet, 10307 Hanoi, Vietnam}

\author{Tran Nam Trung}

\email{tntrung@math.ac.vn}
\address{Institute of Mathematics, VAST, 18 Hoang Quoc Viet, 10307 Hanoi, Vietnam}

\thanks{}
\date{}
\dedicatory{}
\commby{}
%-----------------------------------------------------------
\begin{abstract} We determine the maximum number of maximal independent sets of arbitrary graphs in terms of their covering numbers and we completely characterize the extremal graphs. As an application, we give a similar result for   K\"onig-Egerv\'ary graphs in terms of their matching numbers.
\end{abstract}
% -----------------------------------------------------------
\maketitle
% -----------------------------------------------------------
\section{Introduction}

Let $\m(G)$ be the number of maximal independent sets of a simple graph $G$. Around $1960$, Erd\"os and Moser raised the problem of determining the largest number of  $\m(G)$ in terms of   order of $G$,  which we shall denote by $n$ in this paper, and determining the extremal graphs. In $1965$, Moon and Moser \cite{MM} solved this problem for any simple graph.

This problem now has been  focused in  investigating various classes of graphs such as: for connected graphs by F\"uredi \cite{Fu}; and   independently Griggs et al. \cite{GGG}; for triangle-free graphs by Hujter and Tuza \cite{HuTu} and for connected triangle-free graphs by Chang and Jou \cite{CJ};  Sagan and Vatter \cite{SV} and Goh et al. \cite{YMSV} solved the problem for graphs with at most $r$ cycles; for connected unicyclic graphs by   Koh et al. \cite{KGD}; for trees independently by   Cohen \cite{Co}, Griggs and Grinstead \cite{GG},  Sagan \cite{Sa},  Wilf \cite{W}; for bipartite graphs by Liu \cite{Liu} and bipartite graphs with at least one cycle by Li et al. \cite{LZZ}.

The goal of this paper is to determine the maximum number of $\m(G)$ of arbitrary simple graph $G$ in terms of  its covering number, denoted by $\beta(G)$; and to characterize the extremal graphs.  On that basis we will consequently improve some certain results among those mentioned above. Before stating our results, recall that a matching in $G$ is a set of edges, no two of which meet a common vertex.  The  {\it matching number}  $\nu(G)$ of $G$ is the maximum size of matchings of $G$.   An  induced matching $M$ in a graph $G$ is a matching where no two edges of $M$ are joined by an edge of $G$. The  {\it induced matching number}  $\nu_0(G)$ of $G$ is the maximum size of induced matchings of $G$. We always have $\nu_0(G)\le \nu(G)$; and if $\nu_0(G) =  \nu(G)$ then $G$ is called a {\it Cameron-Walker} graph according to  Hibi et al. \cite{HHKO}.  The main result of the paper is  as follows:

\medskip

\noindent {\bf   Theorem  {\rm (Theorem \ref{upperbound2} and Theorem \ref{CaWa})}. }{\it  Let $G$ be a  graph. Then   $\m(G) \le  2^{\beta(G)}$, and the equality holds  if and only if
   $G$ is     a Cameron-Walker  bipartite graph.
 }
\medskip

A graph $G$ is called a {\it K\"onig-Egerv\'ary graph} if  the matching number is equal to the covering number that is $\beta(G)=\nu(G)$. As an application, we determine the maximum number of $\m(G)$ for K\"onig-Egerv\'ary graphs $G$,  and  characterize the extremal graphs.

\medskip
\noindent {\bf   Corollary \ref{Konig}.} {\it  Let $G$ be a K\"onig-Egerv\'ary graph. Then
$$\m(G) \le 2^{\nu(G)},$$
and the equality  holds  if and only if $G$ is a Cameron-Walker bipartite graph. }
\medskip

It is well-known that all bipartite graphs are K\"onig-Egerv\'ary (see \cite[Theorem 8.32]{BM}).  In general, $\nu(G) \leqslant \lfloor \frac{n}{2} \rfloor$, where $n$ is the order of $G$. Thus Corollary $\ref{Konig}$ improves the main result of Liu (see \cite[Theorem 2.1]{Liu}) for bipartite graphs.
\medskip

\section{Bounds for $m(G)$}

We now recall some basic concepts and terminology from graph theory (see \cite{BM}).  Let $G$ be a simple graph with vertex set $V(G)$ and edge set $E(G)$. An edge $e \in E(G)$ connecting two vertices $x$ and $y$ will be also written as $xy$ (or $yx$).  For a subset $S$ of $V(G)$, we denote by $G[S]$  the induced subgraph of $G$ on the vertex set $S$; and denote $G\setminus S$ by $G[V(G)\setminus S]$. The neighborhood of  $S$ in $G$ is the set  $$N_G(S) := \{y\in V(G) \setminus S\mid xy\in E(G) \text{ for some }x\in S\},$$
the close neighborhood of $S$ is $N_G[S] := S \cup N_G(S)$  and the localization of $G$ with respect to $S$ is  $G_S:=G\setminus N_G[S]$. If $S=\{x\}$,  we write  $N_G(x)$    (resp. $N_G[x]$, $G_x$, $G\setminus x$) instead of   $N_G(\{x\})$ (resp.  $N_G[\{x\}]$, $G_{\{x\}}$, $G\setminus \{x\}$). The number $\deg_G(x):=|N_G(x)|$ is called the {\it degree} of  $x$ in  $G$.
A vertex in $G$ of degree zero is called an {\it isolated vertex} of $G$.   A vertex $x$ of $G$ is called    {\it leaf adjacent to} $y$ if $\deg_G(x) = 1$ and $xy$ is an edge of $G$.   A complete graph with $n$ vertices is denoted by $K_n$.  A graph $K_3$ is called  {\it  triangle}.
 The union of two disjoint graphs $G$ and $H$ is the graph $G \cup H$ with vertex set $V (G \cup H) = V (G) \cup V (H)$ and edge set $E(G \cup H) = E(G) \cup E(H)$. The union of $t$ copies of disjoint graphs isomorphic to $G$ is  denoted by $tG$, where $t$ is a positive integer.

A graph is called {\it totally disconnected} if  it is either  a null graph or containing no edge.  Thus, $\m(G) = 1$ whenever $G$ is  totally disconnected. The following basic lemmas on determining $\m(G)$ for arbitrary graph $G$ will be frequently used later.

\begin{lem} {\rm \cite[Lemma 1]{HuTu}} \label{HuTu} Let $G$ be a graph. Then
\begin{enumerate}
\item  $\m(G) \le \m(G_x) +\m(G\setminus x)$, for any vertex $x$ of $G$.
\item If $x$ is a leaf adjacent to $y$ of $G$, then    $\m(G) = \m(G_x) +\m(G_y)$.
\item If $G_1,\ldots, G_s$ are connected components of   $G$, then
$$\m(G) = \prod_{i=1}^s\m(G_i).$$
\end{enumerate}
\end{lem}

\begin{lem}   \label{W}   If $H$ is an induced subgraph of $G$, then $\m(H) \le  \m(G)$.
\end{lem}

We first give an upper bound for $\m(G)$ in terms of $\nu(G)$, and the extremal graphs. However, that upper bound does not cover the result of Moon and Moser \cite{MM}.

\begin{prop} Let $G$ be a graph. Then, $\m(G)\le 3^{\nu(G)}$  and the equality holds if and only if $G\cong sK_3\cup tK_1$, where $s=\nu(G)$ and $t=|V(G)|-3s$.
\end{prop}
\begin{proof}  We  prove the proposition by induction on $\nu(G)$.  If $\nu(G)=0$, then $G$ is totally disconnected, and then the assertion is trivial.

If $\nu(G) = 1$, let $xy$ be an edge of $G$ and let $S := V(G) \setminus \{x,y\}$. Then $G[S]$ is totally disconnected and if we have two vertices in $S$, say $u$ and $v$, such that $xu$ and $yv$ are edges of $G$, then $u \equiv v$. In particular, there is at most one vertex in $S$ that is adjacent to both $x$ and $y$. We now consider three cases:

\medskip

{\it Case 1}: $x$  and $y$ are not adjacent to any vertex in $S$. In this case, we have $\m(G) = 2$, and the proposition holds.

\medskip

{\it Case 2}: $x$  is not adjacent to any vertex in $S$ and $y$ is adjacent to some vertices in $S$.   Then, we have $\m(G) = 2$, and the  proposition holds.

\medskip

{\it Case 3:} There is a vertex in $S$, say $z$, that is adjacent to both $x$ and $y$. In this case, every other vertex of $S$ is not adjacent to either $x$ or $y$. Thus, $G=K_3\cup tK_1$,  where $t = |V(G)| - 3$ and $\m(G) = 3 = 3^{\nu(G)}$.  Therefore, the proposition is proved in this case.

\medskip

Assume that  $\nu(G) \geqslant 2$. Let $xy$ be an edge of $G$. Since both $x$ and $y$ are not vertices of  the following graphs: $G_x$, $G_y$ and $G\setminus \{x,y\}$,  we deduce that
$$\nu(G_x) \leqslant  \nu(G)-1, \  \nu(G_y) \leqslant  \nu(G)-1 \ \text{ and } \nu(G \setminus \{x,y\}) \le \nu(G)-1.$$
Thus, by the induction hypothesis, we obtain
$$\m(G_x) \leqslant 3^{\nu(G)-1} , \ \m(G_y) \leqslant 3^{\nu(G)-1}    \text{ and }  \ \m(G\setminus \{x, y\}) \leqslant 3^{\nu(G)-1}.$$
Note that $(G\setminus x)_y = G_y$. Combining with Lemma \ref{HuTu}, we obtain
\begin{eqnarray*}
\m(G)&\le& \m(G_x)+\m(G\setminus x) \\
&\le&  \m(G_x)+\m(G_y)+\m(G\setminus \{x, y\}) \\
&\le& 3^{\nu(G)-1} + 3^{\nu(G)-1} + 3^{\nu(G)-1}
=3^{\nu(G)}.
\end{eqnarray*}
This proves the first conclusion of the proposition. The equality $\m(G) = 3^{\nu(G)}$ occurs if and only if
$$\m(G) = \m(G_x)+\m(G\setminus x) , \ \m(G_x)=\m(G_y) = \m(G\setminus \{x,y\})=3^{\nu(G)-1}$$
and
$$\nu(G_x)=\nu(G_y) = \nu(G\setminus \{x,y\})=\nu(G)-1.$$

If $G = sK_3 \cup tK_1$, then $s = \nu(G)$ and $\m(G) = 3^{\nu(G)}$. Therefore, the necessary condition of the second conclusion of the proposition is followed. Now, it remains to prove that if $\m(G) =3^{\nu(G)}$ then $G \cong sK_3 \cup tK_1$.

Indeed, by the induction hypothesis, it follows that  $G_x,  G_y$ and $G\setminus \{x,y\}$ have the same component without isolated vertices, that is $(s-1)K_3$, where $s=\nu(G)$. In particular, $x$ and $y$ are not adjacent to any vertex of $(s-1)K_3$. Let $H$ be an induced subgraph of $G$ on the vertex set $V(G) \setminus V((s-1)K_3)$. Then, $H$ and $(s-1)K_3$ are disjoint subgraphs of $G$. By Lemma \ref{HuTu}, we imply $\m(G) = \m(H)\m((s-1)K_3) = \m(H) 3^{s-1}$. Since $\m(G) = 3^s$, $\m(H) = 3$. Note that $\nu(H) = 1$, so the induction hypothesis  again yields $H = K_3 \cup t K_1$. Thus, $G=sK_3 \cup tK_1$. The proof   is complete.
\end{proof}

 The following lemma gives  a lower bound for $\m(G)$ in terms of the induced matching number $\nu_0(G)$.

\begin{lem}\label{lowerbound} Let $G$ be a   graph. Then, $\m(G)\ge 2^{\nu_0(G)}$.
\end{lem}
\begin{proof} Let  $\{x_1y_1,\ldots,x_ry_r\}$ be an  induced matching of $G$, where  $r=\nu_0(G)$.
Set $H:=G[\{x_1,\ldots,x_r,y_1,\ldots,y_r\}]$.  By Lemma \ref{W}, $\m(G)\ge \m(H)=2^{\nu_0(G)}$.
\end{proof}

Recall that a vertex cover of $G$ is a subset $C$ of $V(G)$ such that for each  $xy \in E(G)$, we have either $x\in C$ or $y\in C$. The {\it covering number}  of $G$,  denoted by $\beta(G)$,  is the  minimum size of  vertex covers   of $G$. From this definition, the following two lemmas  are obvious.

\begin{lem}\label{inducedbeta} Let  $H$ be  an induced subgraph  of $G$. Then,
\begin{enumerate}
\item  If $S$ is a vertex cover of $G$, then $S\cap V(H)$ is a vertex cover of $H$.
\item  $\beta(H)\le \beta(G)$.
\end{enumerate}
\end{lem}

\begin{lem} \label{betaminus}
 Assume  $S$ is a vertex cover of $G$. If $U\subseteq S$, then
  \begin{enumerate}
\item $S\setminus U$ is a vertex cover of $G\setminus U$; and
\item  $\beta(G\setminus U) \le \beta(G)-|U|$.
\end{enumerate}
 \end{lem}

We conclude this section by giving a upper bound for $\m(G)$ in terms of $\beta(G)$.

\begin{thm}\label{upperbound2} Let $G$ be  a graph. Then, $\m(G) \le 2^{\beta(G)}$.
\end{thm}
\begin{proof} We  prove the lemma by  induction on  $\beta(G)$. If $\beta(G)=0$, then $G$ is totally disconnected,  and so the assertion  is trivial.

Assume that $\beta(G)\geqslant 1$. Let $S$ be  a vertex cover  of $G$  such that $|S|=\beta(G)$. Let $x\in S$. By Lemma \ref{betaminus}, we have $\beta(G\setminus x)\le \beta(G)-1$. Hence,  $\m(G\setminus x) \le 2^{\beta(G\setminus x)}$ by the induction hypothesis.

Since $G_x$  is an induced subgraph of $G\setminus x$, $ \m(G_x)\le  \m(G\setminus x)$ by Lemma \ref{W}.  Together with Lemma \ref{HuTu}, we obtain
\begin{eqnarray*}
 \m(G)  &\le&  \m(G\setminus x) + \m(G_x)\\
  &\le&  2.\m(G\setminus x)   \le     2^{\beta(G\setminus x) +1} \le  2^{\beta(G)},
 \end{eqnarray*}
 as required.
 \end{proof}

\section{Extremal graphs}

A graph $G$ is called {\it bipartite} if its vertex set can be partitioned into two
subsets $A$ and $B$ so that every edge has one end in $A$ and one end in $B$; such a
partition is called a {\it bipartition} of the graph, and denoted by $(A, B)$.  If every vertex in $A$ is joined to every vertex in $B$ then $G$
is called a complete bipartite graph, which is denoted by $K_{|A|,|B|}$. A {\it star} is the complete bipartite graph $K_{1,m}$ ($m\ge 0$) consisting of $m + 1$ vertices. A {\it star triangle} is a graph joining some triangles at one common vertex.

Cameron and Walker  \cite{CW} gave a classification of the simple graphs $G$ with $\nu(G) = \nu_0(G)$; such graphs now are the so-called Cameron-Walker graphs (see \cite{HHKO}).

\begin{lem}{\rm (\cite[Theorem 1]{CW} or \cite[p.258]{HHKO})} \label{CaWa_classify} A  graph $G$ is Cameron-Walker      if and only if it is one of the following graphs:
\begin{enumerate}
\item a star;
\item a star  triangle;
\item a finite graph consisting of a connected bipartite graph with  bipartition $(A,B)$ such that there is at least one leaf edge attached to each vertex $i \in A$ and that  there may be possibly some pendant triangles attached to each vertex $j\in B$.
\end{enumerate}
\end{lem}

\begin{exam}\label{exam1} Let $G$ be  Cameron-Walker  graph with $8$ vertices  in Figure 1. Then   $\nu(G)=2$ and
the maximal independent sets of $G$ are
$$ \{1,2,5,6,7,8\}; \{3,4\}; \{3,5,6\}; \{4,7,8\} $$
Hence,  $\m(G)=4 $.
\begin{center}
\includegraphics[scale=0.6]{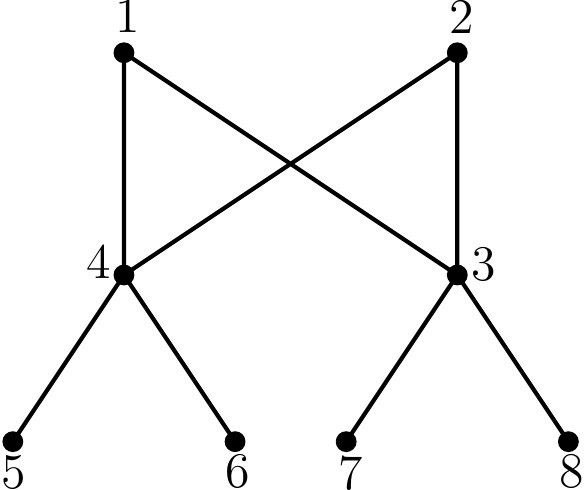}\\
{Figure 1.}
\end{center}
\end{exam}

\begin{thm}\label{CaWa} Let $G$ be a  graph. Then $\m(G) = 2^{\beta(G)}$ if and only if
   $G$ is   a Cameron-Walker  bipartite graph.
\end{thm}
\begin{proof} If $G$ is a  Cameron-Walker bipartite graph, then $\nu_0(G)=\nu(G) = \beta(G)$. Together with Lemma \ref{lowerbound} and Theorem \ref{upperbound2}, this fact yields $\m(G) =  2^{\beta(G)}$.

Conversely, assume that $\m(G) = 2^{\beta(G)}$. We will prove that $G$ is    Cameron-Walker bipartite by induction on $\beta(G)$.

If $\beta(G)=0$, then $G$ is totally disconnected and so the assertion is trivial.   If $\beta(G)=1$, then $G$ is a union of a star and isolated vertices. In this case, $G$ is a Cameron-Walker bipartite graph by Lemma $\ref{CaWa_classify}$.

Assume that  $\beta(G)\ge 2$.  Let $S$ be a  minimal vertex cover of $G$  such that $|S|= \beta(G)$. We first prove two following claims.

\medskip
\noindent{\it Claim 1:} $S$ is an independent set of $G$.

Indeed, assume on the contrary that   there would be an edge, say $xy$, with $x,y\in S$. By Lemma $\ref{inducedbeta}$,  $S \cap V(G_x)$ is a vertex cover of $G_x$. Since $S \cap V(G_x)\subseteq S\setminus \{x,y\}$, we deduce that
$$\beta(G_x) \leqslant |S|-2 =\beta(G)-2.$$
Similarly, $S\setminus\{x\}$ is a vertex cover of $G\setminus x$. Thus $\beta(G\setminus x) \leqslant \beta(G)-1$.

Together those inequalities with Lemma $\ref{HuTu}$ and Theorem $\ref{upperbound2}$, we have
$$\m(G) \leqslant \m(G_x)+\m(G\setminus x) \leqslant 2^{\beta(G)-2} + 2^{\beta(G)-1} < 2^{\beta(G)}.$$
This inequality contradicts our assumption. Therefore, $S$ is an independent set of $G$.

\medskip
\noindent{\it Claim 2:} $\m(G_U) = 2^{\beta(G_U)}$ and $\beta(G_U) = \beta(G)-|U|$ for any $U\subseteq S$.

Indeed, we prove the claim by the induction on $|U|$.  If $|U| = 0$, i.e., $U$ is empty, then there is nothing to prove.

If  $|U|=1$, then $U=\{x\}$ for some vertex $x$. Since $x\in S$, by Lemmas \ref{inducedbeta} and \ref{betaminus}, we have $\beta(G_x)\le\beta(G\setminus x) \le   \beta(G)-1$.  By  Theorem \ref{upperbound2}, $\m(G\setminus x)\le 2^{\beta(G\setminus x)}$ and    $\m(G_x)\le 2^{\beta(G_x)}$. Together these inequalities with equality $\m(G) = 2^{\beta(G)}$, Lemma $\ref{HuTu}$ gives
\begin{eqnarray*}
2^{\beta(G)}=\m(G) &\le& \m(G\setminus x) +\m(G_x)  \le   2^{\beta(G\setminus x)} + 2^{\beta(G_x)}\\
&\le & 2^{\beta(G)-1}+2^{\beta(G)-1}  \le  2^{\beta(G)}.
\end{eqnarray*}
Hence, $\m(G_x) =   2^{\beta(G_x)}$ and $ \beta(G_x) = \beta(G)-1$, and the claim  holds  in this case.

We now assume $|U|\ge 2$. Let $x\in U$ and let  $T:=U\setminus \{x\}$.  Note that  $T$ is a nonempty independent set of $S$ and $|T|=|U|-1$. By the induction hypothesis of our claim, $\m(G_T)=2^{\beta(G_T)}$ and $\beta(G_T)=\beta(G)-|T|$.

Note that, by Claim 1,  $S$ is an independent set of $G$. Thus $S\setminus T =  S\setminus N_G[T]$. By Lemma $\ref{inducedbeta}$, $S\setminus T$ is a vertex cover of $G_T$. Since $x\in S\setminus T$, by the same argument in the inductive step of our claim with $G_T$ replacing by $G$, we have $\m((G_T)_x) = 2^{\beta((G_T)_x)}$ and $\beta((G_T)_x) = \beta(G_T)-1$.

Since $G_U = (G_T)_x$, we obtain $\m(G_U) = 2^{\beta(G_U)}$ and $$\beta(G_U) = \beta(G_T)-1 = \beta(G) -(|T|+1) = \beta(G)-|U|,$$ as claimed.

\medskip

We turn back to the proof of  the theorem. By Claim 1,  $S$ is both a vertex cover and an independent set of $G$. Therefore  $G$ is a bipartite graph with bipartition $(S,V(G)\setminus S)$. It remains to prove $G$ is a Cameron-Walker graph.

 For each $x\in S$, let $U:=S\setminus \{x\}$. By Claim 2, $\beta(G_U) = \beta(G)-|U|=1$. Hence, $G_U$ is a union of  a star with bipartition $(\{x\}, Y)$, where  $\emptyset \ne Y\subseteq V(G)\setminus S$ and isolated vertices.   Thus, there is a vertex $y\in Y$ such that $\deg_{G_U}(y)=1$ and $xy\in E(G)$. Since $V(G)\setminus S$ is an independent set, the equality $\deg_{G_U}(y)=1$ forces $\deg_G(y)=1$. By using  Lemma $\ref{CaWa_classify}$, we conclude that $G$ is a Cameron-Walker graph, and the proof is complete.
\end{proof}

If $G$ is a K\"onig-Egerv\'ary graph, then $\beta(G) = \nu(G)$. Together Theorems $\ref{upperbound2}$ and $\ref{CaWa}$, this fact yields.

\begin{cor}\label{Konig} Let $G$ be a K\"onig-Egerv\'ary graph. Then
$$\m(G) \le 2^{\nu(G)},$$
and the equality  holds  if and only if $G$ is a Cameron-Walker bipartite graph.
\end{cor}

\subsection*{Acknowledgment} This work is partially supported by NAFOSTED (Vietnam), Project  101.04-2015.02.


\begin{thebibliography}{HT}

\bibitem{BM} J. A. Bondy, U. S. R. Murty, Graph Theory, no. 244 in Graduate Texts in Mathematics, Springer, 2008.

\bibitem{CW} K. Cameron and  T. Walker, {\it The graphs with maximum induced matching and maximum matching the same size},
Discrete Math. {\bf 299} (2005) 49-55.

\bibitem{CJ} G. J. Chang and M. J. Jou, {\it  The number of maximal independent sets in connected triangle-free graphs},  Discrete Math.  {\bf  197-198} (1999),  169-178.

\bibitem{Co} D. Cohen, {\it  Counting stable sets in trees}, in Seminaire Lotharingien de combinatoire, 10 session,
R. K\"onig, ed., Institute de Recherche Math\'amatique Avanc\'ee pub., Strasbourg, France, 1984,  48-52.


\bibitem{Fu}  Z. F\"uredi, {\it   The number of maximal independent sets in connected graphs},    J. Graph Theory, {\bf 11} (1987),   463-470.

\bibitem{YMSV} C. Y. Goh, K. M. Koh, B. E. Sagan, and V. R. Vatter, {\it Maximal independent sets in graphs with at most $r$ cycles}, J. Graph Theory,  {\bf 53} (2006) 207-282.

\bibitem{GG} J. R. Griggs and C. M. Grinstead, unpublished result, 1986.

\bibitem{GGG}  J. R. Griggs, C. M. Grinstead and  D. R. Guichard, {\it The number of maximal independent sets in a connected
graph}, Discrete Math. {\bf 68} (1988), 211-220.

\bibitem {HHKO}  T. Hibi, A. Higashitani, K. Kimura and A. B. O'Keefe, {\it  Algebraic study on Cameron–Walker graphs},  J. Algebra {\bf 422} (2015), 257-269.

\bibitem{HuTu} M. Hujter, Z. Tuza, {\it The number of maximal independent sets in triangle-free graphs}, SIAM J. Discrete
Math. {\bf 6} (1993) 284-288.

\bibitem{KGD}  K. M. Koh, C. Y. Goh and F. M. Dong, {\it The maximum number of maximal independent sets in unicyclic connected graphs}, Discrete Math. {\bf 380} (2008), 3761-3769.

\bibitem{LZZ} S. Li, H. Zhang and  X. Zhang, {\it Maximal independent sets in bipartite graphs with at least one cycle}, Discrete Math. and Theor. Comput. Sci.  {\bf 15} (2013), 243-248.

\bibitem{Liu}    J. Liu, {\it Maximal independent sets in bipartite graphs}, J. Graph Theory, {\bf 17} (1993),   495-507.

\bibitem{MM} J. W. Moon and  L. Moser, {\it On cliques in graphs}, Israel J. Math., {\bf 3} (1965), pp. 23-28.


\bibitem{Sa}  B. E. Sagan, {\it A note on independent sets in trees}, SIAM J. Discrete Math. {\bf 1} (1988), 105-108.

\bibitem{SV} B. E. Sagan, V. R. Vatter, {\it Maximal and maximum independent sets in graphs with at most $r$ cycles}, J. Graph Theory {\bf 53} (2006) 283-314.

 \bibitem{W}  H. S. Wilf, {\it The number of maximal independent sets in a tree}, SIAM J. Algebraic Discrete
Methods {\bf 7} (1986), 125-130.


 \end{thebibliography}
 \end{document}